\documentclass[a4paper, 11pt]{article}

\usepackage{amsthm}
\usepackage{amsfonts}
\usepackage{amsmath}
\usepackage{latexsym}
\usepackage{amssymb}
\usepackage{url}
\usepackage{enumerate}
\usepackage[active]{srcltx}
\usepackage{mathrsfs}
\usepackage{array}
\usepackage{cite}
\usepackage{marvosym}

\newtheorem{theorem}{Theorem}

\newtheorem{proposition}{Proposition}

\newtheorem{definition}{Definition}

\DeclareMathOperator{\cone}{cone}
\DeclareMathOperator{\inte}{int}

\newcommand{\lng}{\lf\langle}
\newcommand{\rng}{\rg\rangle}
\newcommand{\lf}{\left}
\newcommand{\rg}{\right}

\newcommand{\R}{\mathbb R}

\newcommand{\bs}{\begin{smallmatrix}}
\newcommand{\es}{\end{smallmatrix}}
\newcommand{\tp}{^\top}

\newcommand{\pdflatex}

\begin{document}
\title{Conic optimization and complementarity problems}
\author{S. Z. N\'emeth\\
 University of Birmingham, School of Mathematics,
 \\Watson Building, Edgbaston, Birmingham B15 2TT,\\
 \and Guohan Zhang\\
University of Birmingham, School of Mathematics,
 \\Watson Building, Edgbaston, Birmingham B15 2TT }

\maketitle

\begin{abstract}
	Although the Karush-Kuhn-Tucker conditions suggest a connection between a conic optimization problem and a complementarity problem, it
	is difficult to find an accessible explicit form of this relationship in the literature. This note will present such a relationship.
\end{abstract}

\section{Introduction}

Although the Karush-Kuhn-Tucker (KKT) conditions suggest a connection between constrained optimization and complementarity problems, it
is difficult to find this connection written explicitly and explained in a perspicuous way, easily accessible to beginners of the field as well.
The connection is more in the domain of the mathematical folklore, assuming that it should be clear that the complementary slackness condition
corresponds to a complementarity problem. The aim of this short note is to present this duality between optimization problems and complementarity
problems in a more clear-cut way. Due to the recent development of conic optimization and the applications of cone-complementarity problems, it 
is desirable to make this connection for more general cones, while still keeping it accessible to a wider audience. Especially because
apparently all applications of cone-complementarity problems defined by cones essentially different from the nonnegative orthant are based on 
this correspondence. There are several such applications in physics, mechanics, economics, game theory, robotics 
\cite{ZhangZhangPan2013,YonekuraKanno2012,AghassiBertsimas2006,NishimuraHayashiFukushima2012,LuoAnXia2009,NishimuraHayashFukushima2009,AndreaniFriedlanderMelloSantos2008,KoChenYang2011,ChenTseng2005}.

The concept of complementarity occurs naturally in the dual of an optimization problem. The simplest such complementarity 
property occurs in the dual of a classical solvable linear programming problem \cite{BTN2001}. Define a linear programming problem as
\begin{equation}\label{LPprimal}\tag{$\mathcal{P}$}
\begin{array}{l}
	\min\text{ }c^T x  \\
	\text{subject to }  Ax \geq b,
\end{array}
\end{equation}
while its dual problem as
\begin{equation}\label{LPdual}\tag{$\mathcal{D}$}
	\begin{array}{ll}
		\max\text{ }b^T y  \\
		\text{subject to }  & A^Ty = c, \\
				    & y \in \mathbb{R}^n_+. 
\end{array}
\end{equation}
Then \eqref{LPprimal} is solvable if and only if \eqref{LPdual} is solvable. The duality gap
\begin{equation*}
c^Tx- b^Ty = y^TAx -y^Tb=y^T(Ax-b) \geq 0
\end{equation*}
will be nonnegative for any feasible $(x,y)$ . If $x^*$, $y^*$ are feasible solutions of \eqref{LPprimal} and \eqref{LPdual}, respectively, then 
their optimality is equivalent to the complementary slackness, that is, to
\begin{equation*}
{y^*}^T(Ax^* -b )=0.
\end{equation*}
An extension to this property can be obtained by replacing $\geq$ in the constrained condition with $\geq_K 0$ induced by some
cone $K$. This change give rise to the linear conic programming problem defined by
\begin{equation}\label{LCPprimal}\tag{$\mathcal{CP}$}
	\begin{array}{l}
		\min\text{ }c^T x  \\
		\text{subject to }  Ax-b \succeq_K 0,
	\end{array}
\end{equation}
and with the dual
\begin{equation}\label{LCPdual}\tag{$\mathcal{CD}$}
	\begin{array}{ll}
		\max\text{ }b^T y & \\
		\text{subject to } & A^T y = c, \\
				   & y  \succeq_{K^*} 0.
	\end{array}
\end{equation}
The complementary slackness property from linear programming needs some adjustment to remain valid in this case. First, the strict feasibility 
and boundedness form below (from above) of the primal problem (dual problem) implies the solvability of the dual problem (primal problem). In 
this case the optimal values of the two problems are equal to each other. With these properties in mind we formulate now the complementary 
slackness property for the linear conic programming problem: Assume that at least one of the problems \eqref{LCPprimal} and \eqref{LCPdual}
is bounded and strictly feasible. Then, the optimality of a feasible pair $(x^*,y^*)$ of solutions of \eqref{LPprimal} and \eqref{LPdual}, 
respectively, is equivalent to the complementary slackness, that is, to
\begin{equation*}
{y^*}^T(Ax^* -b )=0.
\end{equation*}
When $K=\mathbb{R}^n_+$, the problems \eqref{LCPprimal}, \eqref{LCPdual} reduce to \eqref{LPprimal}, \eqref{LPdual}, with the corresponding
stronger complementary slackness property, which in general does not hold. After this ``appetizer'' about the connections between optimization 
and complementarity let us proceed to presenting our main result, that is Theorem \ref{mth}, which will exhibit a more explicit connection 
between general conic optimization problems and complementarity problems. The main result will be based on some preliminary concepts and
properties presented in the next section. We will use a variant of the Karush-Kuhn-Tucker theorem \cite{KKT1951} presented in 
\cite{RA2006} about the complementarity property of the optimal points under Slater's condition.
 

\section{Preliminaries}

Consider the $n$-dimensional Euclidean space $\R^n$ whose elements are assumed to be $n\times 1$ column vectors and its canonical scalar product 
$\lng\cdot,\cdot\rng$ is defined by $\lng x,y\rng=x\tp y$. The vector $x$ is called \emph{perpendicular} to the vector $y$ if $\lng x,y\rng=0$,
which will also be denoted by $x\perp y$.

In this section we will recall the standard notions of a cone, dual of a cone, complementarity problem, mixed complementarity problem and conic
optimization problem in $\R^n$.


Recall that a \emph{closed convex cone} is closed convex set which is invariant under the multiplication of vectors by positive scalars. For 
simplicity in this note a closed convex cone will be simply called \emph{cone}. The \emph{dual} of a cone $K$ is the cone $K^*$ defined by

\begin{equation*}
K^*=\{d \in \mathbb{R}^n\mid v^Td \geq 0, \text{ } \forall v \in K \}
\end{equation*}

The cone \[\R^n_+=\{x=(x_1,\dots,x_n)\tp\in\R^n\mid x_1\ge0,\dots,x_n\ge0\}\] is called the \emph{nonnegative orthant} and $(\R^n_+)^*=\R^n_+$.

The cone $K$ induces the relations $\le_K$ and $\ge_K$ defined by $x\le_K y\iff y-x\in K$ and  $x\ge_K y\iff x-y\in K$, respectively. If 
$K=\R^n_+$, then the relations  $\le_K$ and $\ge_K$ will simply be denoted by $\le$ and $\ge$, respectively. The relation $\le_K$ is reflexive,
transitive and compatible with the linear structure of $\R^n$, that is, $x\le_K y$ implies $\lambda x+z\le_K\lambda y+z$, for any $\lambda>0$
real number and any $z\in\R^n$.

For an arbitrary set $C\subseteq\R^n$ we will denote by $\cone(C)$ the smallest cone containing $C$. More precisely, $\cone(C)$ is defined by 
\begin{enumerate}
	\item $C\subseteq\cone(C)$,
	\item $\cone(C)\subseteq K$, for any cone $K$ with $C\subseteq K$.
\end{enumerate}
The cone $\cone(C)$ is called the cone \emph{generated by $C$}. 

\begin{definition}\label{CP}
	Given a cone $K$ and a mapping $F:\R^n\to\mathbb{R}^n$, the \emph{complementarity problem} $CP(K,F)$ is to find a vector 
	$x \in \mathbb{R}^n$ which satisfies the following conditions:
		\begin{equation}
			K\ni x\perp F(x)\in K^*.
		\end{equation}
\end{definition}

If $K=\R^n_+$, then $CP(K,F)$ will be simply denoted by $CP(F)$.


\begin{definition}\label{micp}
	Let $G$ and $H$ be two mappings from $\R^p\times\R^q$ into $\R^p$ and $\R^q$, respectively, where $p+q=n$ and let $C\subset\R^q$ be a 
	cone. The \emph{mixed complementarity problem} $MiCP(G,H,C,p,q)$ is to find a pair of vectors $(u,v)\in\mathbb{R}^{p} \times C$ such that
	\begin{equation}
		\begin{aligned}
			& G(u, v) = 0,
			& C\ni v \perp H(u, v) \in C^*.
		\end{aligned}
	\end{equation}
\end{definition}

If $C=\R^q_+$, then we simply denote $MiCP(G,H,C,p,q)$ by\linebreak $MiCP(G,H,p,q)$, which is the classical mixed complementarity problem defined
in \cite{FP2003}.

It can be easily seen that the \emph{mixed complementarity problem}\linebreak $MiCP(G,H,C,p,q)$ is equivalent to the complementarity problem 
$CP(F,K)$, where $F:\R^n\to\R^n$ is defined by $F(x)=(G(x),H(x))$ and $K=\R^p\times C$ (with the trivial identification between $\R^p\times\R^q$
and $\R^{p+q}$). We also remark that any complementarity problem $CP(T,C)$ in $\R^q$ can be viewed as the mixed complementarity problem 
$MiCP(G,H,C,p,q)$, where $G:\R^p\times\R^q\to\R^p$ is the identically zero mapping and $H:\R^p\times\R^q\to\R^q$ is defined by $H(u,v)=T(v)$.

\begin{definition}
Given a convex set $X$ and a point $x \in X$ the set
\begin{equation}
K_X(x) = \cone(X-x)
\end{equation}
 is called the \emph{cone of feasible directions} of $X$ on $x$.
\end{definition}

\begin{definition}
Consider a closed and convex set $X\subseteq\mathbb{R}^n $ and a point $x\in X$. The set
\begin{equation}
	N_X(x)=-[\cone(X - x)]^*
\end{equation}
is called the \emph{normal cone} of $X$ at $x$.
\end{definition}
It is easy to conclude from the definition that $v \in N_X(x)$ if and only if 
\begin{equation}
\langle v, y-x \rangle \leq 0
\end{equation}
for all $y \in X$. Moreover, if $X$ is a closed convex cone in $\mathbb{R}^n$, then
\begin{equation}\label{ncone}
N_X(x)=
(-X^*)\cap x^\perp
\end{equation}
where $x^\perp=\{y\in\R^n:y\perp x\}$ denotes the orthogonal complement of $x$. Indeed, if 
$v\in(-X^*)\cap x^\perp$, for any $y \in X$, then 
$\langle v, y-x \rangle = \langle v, y \rangle - \langle v, x \rangle = \langle v, y \rangle \leq 0$, hence $v\in N_X(x)$.

Conversely, if $v \in N_X(x)$, then by taking $y=(1/2)x \in X$ and $y=2x \in X$, we get 
$\langle v, x \rangle \leq 0 \leq \langle v, x \rangle$, so $v \perp x$. Thus, for any $y \in X$,$\langle v, y-x \rangle = \langle v, y \rangle
\leq 0$, hence $v \in -X^*$. In conclusion, $v\in (-X^*)\cap x^\perp$. 

Consider the nonlinear optimization problem 
\begin{equation}\label{nlp}
	\begin{array}{lcl}
\min f(x) && \\
\text{subject to} & g_i(x) \leq 0,\text{ }i=1,\dots,m, & \\
& h_i(x) = 0,\text{ }i = 1, \dots , p, & \\
& x \in X_0 ,&
\end{array}
\end{equation}
where $p>0$, the function $f : \mathbb{R}^n \mapsto \mathbb{R}$, $g_i : \mathbb{R}^n \mapsto \mathbb{R}$, $i= 1, \dots, m$, and 
$h_i : \mathbb{R}^n \mapsto \mathbb{R}$, $i= 1, \dots, p$ are continuously differentiable, and that the set $X_0 \subseteq \mathbb{R}^n$ is 
convex and closed.

Before stating the next theorem, we need to recall Slater's condition. We say that \emph{Slater's condition} hold for problem \eqref{nlp} if 
there exists a point $x^s \in X_0$ such that $g_i(x^s) <0$, $i =1, \dots, m$, $h_i(x^s) =0$, $i =1, \dots, p$, and $x^s \in \inte X_0$.

\begin{theorem}\emph{\cite{RA2006}} Assume that $\hat{x}$ is a local minimum of problem \eqref{nlp}, the function $f$ is continuous at some feasible point $x_0$, and
	Slater's condition is satisfied. Then there exist $\hat{\lambda} \in \mathbb{R}^n_+$ and $\hat{\mu} \in \mathbb{R}^p$ such that
	\begin{equation}\label{kkts1}
		0 \in \nabla f(\hat{x}) + \sum_{i=1}^m \hat{\lambda}_i \nabla g_i(\hat{x}) + \sum_{i=1}^p \hat{\mu}_i \nabla h_i(\hat{x})+ 
		N_{X_0}(\hat{x})
	\end{equation}
	and
	\begin{equation}\label{kkts2}
		\hat{\lambda}_i \nabla g_i(\hat{x})=0 , i =1, \dots, m. 
	\end{equation}
	Conversely, if for some feasible point $\hat{x}$ of \eqref{nlp} and some $\hat{\lambda} \in \mathbb{R}^m_+$ and $\hat{\mu} \in
	\mathbb{R}^p$ conditions \eqref{kkts1} and \eqref{kkts2} are satisfied, then $\hat{x}$ is the global minimum of problem \eqref{nlp}.
\end{theorem}

In order to prove our main theorem (Theorem \ref{mth}), we need to state the following classical result of convex optimization.
\begin{proposition}\label{optgrad}
	The continuously differentiable function $f:\mathbb{R}^n\to\mathbb{R}$ is convex if and only if
	\begin{equation}\label{descentgrad}
		f(y) \geq f(x) + \lng\nabla f(x), y-x\rng \text{, }\forall  x, y \in \mathbb{R}^n.
	\end{equation}
	Moreover, if $f$ is convex, then $x$ is a minimizer of $f$ if and only if
	\begin{equation}\label{minigrad}
		\lng\nabla f(x), y-x\rng \geq 0 \text{, } \forall y \in \mathbb{R}^n.
	\end{equation}
\end{proposition}

\begin{definition}
	Let $f: \mathbb{R}^q \mapsto \mathbb{R}$ be a function, $K\subset\mathbb{R}^m$ be a cone, $A$ a $p\times q$ matrix and 
	$b\in \mathbb{R}^p$. Then, the problem
	\begin{equation}
		CO(f, A, b , K, p, q): \left\{
		\begin{array}{ll}
			\min f(x) & \\
			\text{subject to } & Ax=b,\\
					   & x\in K.
		\end{array}
		\right.
	\end{equation}
	is called \emph{conic optimization problem}.
\end{definition}
\section{The main result}

In the previous sections, we stated the complementarity problems and the complementarity relation in linear (conic) programming problems. We also presented the Karush-Kuhn-Tucker condition which illustrated the properties of optimal solutions. Based on these results, we will prove the equivalence of a conic optimization problem with a mixed complementarity problem. 
\begin{theorem}\label{mth}
	Let $f: \mathbb{R}^q \mapsto \mathbb{R}$ be a differentiable convex function at $\hat{x} \in \mathbb{R}^q \setminus \{0\}$, $K \subseteq
	\mathbb{R}^q$ be a conic set with smooth boundary, $A$ is a $p \times q$ matrix of full rank and $b \in \mathbb{R}^p$. Suppose that the
	intersection of the interior of $K$ and the linear subspace $\{x \in \mathbb{R}^q : Ax=b\}$ is nonempty. Then, $\hat{x}$ is a solution of 
	$CO(f,A,b,K)$ if and only if $(\hat{y},\hat{x})$ is a solution of $MiCP(G,H,K, p, q)$, where $G(y,x)=b-Ax$, 
	$H(y,x)=\nabla f(x)-A^Ty$, which can be
	written explicitly as
	\begin{equation*}
		A\hat{x}= b, K \ni \hat{x} \perp \nabla f(\hat{x})- A^T\hat{y} \in K^*.
	\end{equation*}
\end{theorem}
\begin{proof}
Let $\hat x$ be a solution of $CO(f, A, b, K)$ and consider the the preceding theorem with $X_0=K$, $h(x)= b-Ax$ and $\hat\lambda = \hat y$. 
Then, equation
\eqref{kkts1} becomes:
\begin{equation}\label{eno}
	0 \in \nabla f(\hat x) - A^T\hat y + N_K(\hat x)
\end{equation}
By \eqref{eno} and \eqref{ncone}, we have that
\begin{equation*}
	\nabla f(\hat x) - A^T\hat y \in K^* \text{ and } \nabla f(\hat x) - A^T\hat y \perp\hat x.
\end{equation*}
Hence, since $\hat x\in K$ and $A\hat x = b$, it follows that $(\hat y, \hat x)$ is a solution of $MiCP(G, H, K)$.

Conversely, suppose that $(\hat y, \hat x)$ is a solution of 
$MiCP(G, H, K)$. For any feasible solution $x$ in $CO(f, A, b, K)$,  we have:
\begin{eqnarray}\label{cccc}
0 \leq \langle \nabla f(\hat x) - A^T\hat y, x \rangle= \langle \nabla f(\hat x) - A^T\hat y, x-\hat x \rangle\nonumber\\ =\langle \nabla f(\hat
x) , x-\hat x \rangle - \langle A^T\hat y, x-\hat x \rangle
\end{eqnarray}
Because $\langle A^Ty, x-\hat x \rangle = \langle y, Ax-A\hat x \rangle=\langle y,b-b\rangle=0$, by the convexity of $f$, inequality 
\eqref{cccc} and Proposition \ref{optgrad}, we have
\begin{equation*}
	0 \leq \langle \nabla f(\hat x), x-\hat x \rangle \leq f(x) - f(\hat x).
\end{equation*}
Hence, $f(\hat x) \leq f(x)$ for any feasible $x$. Therefore, $\hat x$ is a solution of \linebreak $CO(f, A, b, K)$.
\end{proof}

\section{Final remarks}

This short note presented an explicit connection between conic optimization and complementarity problems, connection which comes from the 
complementary slackness relation of the Karush-Kuhn-Tucker conditions. Although the complementary slackness suggests that such a
connection should exist, it is difficult to find it explicitly in the literature. Hopefully, this short note will be a useful reference for
some readers.



\bibliographystyle{plain}
\bibliography{RefCC}
\end{document}